\documentclass[runningheads,a4paper]{llncs}

\usepackage{graphicx}
\usepackage{wrapfig}
\usepackage[all]{xy}
\usepackage{latexsym}
\usepackage{amssymb}
\usepackage{amsmath}
\usepackage{txfonts}
\usepackage{etex}
\usepackage{comment}
\usepackage{bussproofs}
\usepackage{braket}

\newcommand{\Sets}{\mathbf{Set}}

\newcommand{\Meas}{\mathbf{Meas}}

\newcommand{\inverse}[1]{{#1}^{\hspace{-0.15em} - \hspace{-0.1em}1}\hspace{-0.1em}}

\newcommand{\lrceil}[1]{\left\lceil {#1} \right\rceil}

\newcommand{\lrangle}[1]{\left\langle {#1} \right\rangle}
\newcommand{\lrAngle}[1]{\left\langle\hspace{-0.2em}\left\langle {#1} \right\rangle\hspace{-0.2em}\right\rangle}

\begin{document}
\title{The Giry Monad is not Strong for the Canonical Symmetric Monoidal Closed Structure on $\Meas$}
\author{Tetsuya Sato}
\institute{
	Research Institute for Mathematical Sciences,
	Kyoto University,
	Kyoto, 606-8502, Japan\\
	\email{satoutet@kurims.kyoto-u.ac.jp}
}
\maketitle
\begin{abstract}
We show that the Giry monad is not strong with respect to the canonical symmetric monoidal closed structure on the category $\Meas$ of all measurable spaces and measurable functions.
\end{abstract}
\section{Introduction}
A motivation of this research is to give a denotational semantics of Higher-order continuous probabilistic programming language.
The denotational semantics of discrete probabilistic programming language is categorified by using the (sub-)distribution monad on the category $\Sets$ of all sets and functions.
The categorified semantics supports the Higher-order functions since the category $\Sets$ is cartesian closed, and the (sub-)distribution monad is commutative strong with respect to the cartesian monoidal structure.
On the other hand, denotational semantics of continuous first-order probabilistic programming language is categorified by using the (sub-probabilistic) \emph{Giry monad}.
The Giry monad is a monad on the category $\Meas$ of measurable spaces and functions, which introduced by Giry to give a categorical definition of continuous random processes such as (Labelled) Markov processes in the paper \cite{Giry1982}.
The Giry monad is commutative strong with respect to the cartesian monoidal structure of $\Meas$, and hence it supports first-order semantics for continuous probabilistic language.
However, it does not support Higher-order functions because the category $\Meas$ is not cartesian closed \cite{aumann1961}.

To give a categorical semantics of higher-order continuous probabilistic programming language, 
we have to find a monoidal closed structure which supports continuous probabilistic processes/calculations.

In fact, there is a canonical \emph{symmetric monoidal closed structure} on $\Meas$ that is defined by the finest $\sigma$-algebra $\Sigma_{X \otimes Y}$ over product sets $|X|\times |Y|$ that makes all constant graph functions measurable (Section \ref{SMCC_Meas}).
If the Giry monad was strong with respect to it then we obtained a categorical semantics of higher-order continuous probabilistic programming language.
However, unfortunately, the Giry monad \emph{is not strong} with respect to the canonical symmetric monoidal closed structure.

In this paper, we prove that Giry monad is not strong with respect to the canonical symmetric monoidal closed structure as follows:
We recall that a strength of a monad with respect to a symmetric monoidal closed category corresponds bijectively to a tensorial strength \cite{Kock1972}.
We show that a tensorial strength for any monad on $\Meas$ with respect to the canonical symmetric monoidal closed structure is uniquely determined if exists (Section \ref{Uniqueness_strength}).
This implies that there is a unique candidate of the strength of Giry monad with respect to the canonical symmetric monoidal closed structure.
We give a counterexample that the candidate associates a non-measurable function to some pair of measurable spaces (Section \ref{Giry_not_Strong}).
\subsection{Preliminaries}
We refer the definitions of monads, monoidal categories, and monoidal functors from \cite{citeulike:423557}, and refer the definition of strong monads on symmetric monoidal closed categories from \cite{Kock1970,Kock1972}.
The notion of tensorial strength can be relaxed to a monad on a symmetric monoidal category.
We often call monads equipped with tensorial strengths strong monads (see \cite[Definition 3.2]{MOGGI199155} or \cite[Section 7.1]{Goubault-Larrecq2002}).

Throughout this paper, we use the category $\Meas$ of all measurable spaces and measurable functions.
The category $\Meas$ is complete and cocomplete.
Hence we enjoy the cartesian monoidal structure $(\Meas, \times, 1)$.
Moreover, it is a \emph{topological category} \cite{BRUMMER198427,HERRLICH1974125}.
We emphasise that the category $\Meas$ is \emph{not cartesian closed} because there is no $\sigma$-algebra over $\Meas([0,1],2)$ satisfying the axioms of exponential object \cite{aumann1961}.

We also introduce the following notations on measure theory:
\begin{itemize}
\item For each measurable spaces, we denote by $|X|$ and $\Sigma_X$ the underlying set and $\sigma$-algebra of $X$ respectively.
\item The indicator function $\chi_A \colon X \to \mathbb{R}$ of a subset $A$ of $X$ is defined by $\chi_A(x) = 1$ ($x \in A$) and $\chi_A(x)=0$ ($x \notin A$).
Note that $\chi_A$ is measurable if and only if the corresponding subset $A$ is measurable (i.e. $A \in \Sigma_X$).
\end{itemize}

We recall that the Borel $\sigma$-algebra $\mathcal{B}(\mathbb{R})$ over the real line $\mathbb{R}$ is generated by the family of all half-open intervals $\Set{\left[\alpha,\beta\right)| \alpha,\beta \in \mathbb{R}}$.
We remark that each singleton $\{r\}$ is a Borel set, because $\{r\} = \bigcap_{n \in \mathbb{N}} \left[r,r + 1/(n+1) \right)$, and hence any countable subset of $\mathbb{R}$ is a Borel set.
By the Caratheodory's extension theorem, there is a unique measure on $\mathcal{B}(\mathbb{R})$ assigning $m(\left[\alpha,\beta\right)) = \beta - \alpha$.
We denote it by $m$, and call it \emph{the Borel measure} on the real line $\mathbb{R}$.
We remark that $m(\{r\}) = 0$ for any $r \in \mathbb{R}$, and hence $m(A) = 0$ for any countable subset $A \subseteq \mathbb{R}$.
\subsection{The Giry monad}
%
The Giry monad \cite{Giry1982} is a monad on the category $\Meas$ that is introduced by Giry, which captures continuous/non-discrete probabilistic computations such as Labelled Markov processes.
For example, Markov processes are arrows in the Kleisli category of Giry monad, and the Chapman-Kolmogorov equation for Markov processes is characterised as associativity of multiplications of the Giry monad.

The structure of Giry monad $\mathcal{G}$ is defined as follows:
\begin{itemize}
\item For any measurable space $X$,
the measurable space $\mathcal{G}X$ is defined by
\begin{itemize}
\item the underlying set $|\mathcal{G}X|$ is \emph{the set of probability measures} on $X$.
\item the $\sigma$-algebra $\Sigma_{\mathcal{G}X}$ is the \emph{coarsest} one over $|\mathcal{G}X|$
that makes the evaluation function $\mathrm{ev}_A \colon \mathcal{G}X \to [0,1]$ defined by $\nu\mapsto\nu(A)$ measurable for any $A \in \Sigma_X$, where $\Sigma_{[0,1]}$ is the Borel $\sigma$-algebra over the unit interval, which is introduced in the same way as $\mathcal{B}(\mathbb{R})$.
\end{itemize}
\item For each $f \colon X \to Y$ in $\Meas$, $\mathcal{G}f \colon \mathcal{G}X \to \mathcal{G}Y$ is defined by $(\mathcal{G}f)(\mu) = \mu(\inverse{f}(-))$.
\item The unit $\eta$ is defined by $\eta_X (x) = \delta_{x}$,
where $\delta_{x}$ is the \emph{Dirac measure} centred on $x$. 
\item The Kleisli lifting
of $f \colon X \to \mathcal{G}Y$
is given by $f^\sharp(\mu)(A) = \int_{X} f({-})(A)~d\mu$ ($\mu \in \mathcal{G}X$).
\end{itemize}
We also consider the subprobabilistic variant $\mathcal{G}_{\mathrm{sub}}$ of the Giry monad; 
the underlying set $|\mathcal{G}_{\mathrm{sub}}X|$ is the set of \emph{subprobability} measures on $X$.

Both the Giry monad $\mathcal{G}$ and its subprobabilistic variant $\mathcal{G}_{\mathrm{sub}}$ is strong and commutative with respect to the cartesian monoidal structure on $\Meas$ in the sense of \cite{MOGGI199155}.
The tensorial strength $\mathrm{st}^{\mathcal{G}{\times}}_{-,=} \colon ({-})\times \mathcal{G}({=}) \Rightarrow \mathcal{G}({-} \times {=})$ is given by the product measure $\mathrm{st}_{X,Y}(x,\nu) = \delta_x \times \nu$.
The commutativity is shown from the Fubini theorem, and the double strength $\mathrm{dst}^{\mathcal{G}{\times}}_{-,=} \colon \mathcal{G}({-})\times \mathcal{G}({=}) \Rightarrow \mathcal{G}({-} \times {=})$ is given by $\mathrm{dst}_{X,Y}(\nu_1,\nu_2) = \nu_1 \times \nu_2$.
\section{The Canonical Symmetric Monoidal Closed Structure on $\Meas$}\label{SMCC_Meas}
The category $\Meas$ is not cartesian closed, but there is the following canonical symmetric monoidal \emph{closed} structure on $\Meas$ (see also \cite{Culbertson_2012, 2016arXiv160102593S}).
We first consider the following two families of constant graph functions:
\begin{itemize}
\item  $\Gamma_{x} \colon |Y| \to  |X| \times |Y|$ defined by $\Gamma_{x}(y) = (x,y)$ for any $y \in |Y|$ ($x \in |X|$).
\item  $\Gamma_{y} \colon |X| \to  |X| \times |Y|$ defined by $\Gamma_{y}(x) = (x,y)$ for any $x \in |X|$ ($y \in |Y|$).
\end{itemize}
Next, we introduce the following symmetric monoidal closed structure on $\Meas$:
\begin{itemize}
\item
The monoidal product functor $\otimes$ is defined by $X \otimes Y = (|X| \times |Y|, \Sigma_{X \otimes Y})$ where the $\sigma$-algebra $\Sigma_{X \otimes Y}$ is the finest $\sigma$-algebra $\Sigma$ such that 
\begin{itemize}
\item $\Gamma_{x}$ is a measurable mapping $Y \to  (|X| \times |Y|, \Sigma)$ for any $x \in X$, and 
\item $\Gamma_{y}$ is a measurable mapping $X \to  (|X| \times |Y|, \Sigma)$ for any $y \in Y$,
\end{itemize}
\item
The internal Hom functor $\multimap$ is defined by $(X \multimap Y) = (\Meas(X,Y),\Sigma_{X \multimap Y})$ where the $\sigma$-algebra $\Sigma_{X \multimap Y}$ is the coarsest one generated by 
\[
\lrAngle{x, U} =\Set{f \in \Meas(X,Y) | f(x) \in U} \quad (x \in |X|, U \in \Sigma_Y).
\]
\end{itemize}
We remark that the forgetful functor $|-|\colon \Meas \to \Sets$ forms a strict symmetric monoidal functor from $(\Meas, \otimes, 1)$ to $(\Sets, \times, 1)$.
\begin{lemma}
The currying operation forms a natural isomorphism
$\Meas(X \otimes Y,Z) \simeq \Meas(X, Y \multimap Z)$ for all measurable spaces $X,Y,Z$.
\end{lemma}
\begin{proof}
Let $f$ be an arbitrary function of type $|X|\times |Y| \to |Z|$.
The curried function $\lrceil{f}$ is then a function of type $|X| \to \Sets(|Y|, |Z|)$.
The currying operator $\lrceil{-}$ is obviously natural and isomorphic as a transformation on just functions.
Hence, it suffices to show that the original $f$ is measurable if and only if the curried $\lrceil{f}$ returns measurable functions, and is measurable itself.
\begin{align*}
\lefteqn{f \in \Meas(X \otimes Y,Z)}\\
&\iff \forall{V \in \Sigma_Z}. \inverse{f}(V) \in \Sigma_{X \otimes Y}\\
&\iff \forall{V \in \Sigma_Z}.\left(
(\forall{x \in X}.\inverse{\Gamma_x}(\inverse{f}(V))\in\Sigma_Y)\wedge
(\forall{y \in Y}.\inverse{\Gamma_y}(\inverse{f}(V))\in\Sigma_X)\right)\\
&\iff \forall{V \in \Sigma_Z}.\left(
(\forall{x \in X}.\inverse{(\lrceil{f}(x))}(V)\in\Sigma_Y)\wedge
(\forall{y \in Y}.\inverse{\lrceil{f}}\lrAngle{y,V}\in\Sigma_X)\right)\\
&\iff
(\forall{x \in X}.\lrceil{f}(x) \in \Meas(Y,Z)) \wedge
(\forall{V \in \Sigma_Z}.\forall{y \in Y}.\inverse{\lrceil{f}}\lrAngle{y,V} \in\Sigma_X)\\
&\iff \lrceil{f} \in \Meas(X,Y\multimap Z)
\end{align*}
\end{proof}
We remark that the uncurried mapping of the identity mapping $\mathrm{id}_{X \multimap Y} \colon X \multimap Y \to X \multimap Y$ on $X \multimap Y$ is called the evaluation mapping $\mathrm{ev}_{X,Y} \colon (X \multimap Y) \otimes X \to Y$.

This construction is similar to the classical symmetric monoidal closed structure $(\mathbf{Top},\otimes,1)$ on the category $\mathbf{Top}$ of topological spaces and continuous functions (see \cite[Example 6.1.9.g]{opac-b1126838}, \cite[Section 3]{Wischnewsky1974}, and \cite[Remark 6.4]{kelly2009}).

By \cite[Proposition 3.1]{kelly2009}, any symmetric monoidal closed structure on $\Meas$ is isomorphic to a symmetric monoidal closed structure $(\mathbin{\dot{\otimes}}, 1,\mathbin{\dot{\multimap}})$ such that the forgetful functor $|-|\colon \Meas \to \Sets$ is strictly symmetric monoidal functor from $(\Meas, \mathbin{\dot{\otimes}}, 1)$ to $(\Sets, \times, 1)$, and each internal Hom $X \mathbin{\dot{\multimap}} Y$ is a measurable space $(\Meas(X,Y), \Sigma_{X \mathbin{\dot{\multimap}} Y})$ for some $\sigma$-algebra $\Sigma_{X \mathbin{\dot{\multimap}} Y}$ over $\Meas(X,Y)$.

Consider an arbitrary symmetric monoidal closed structure $(\mathbin{\dot{\otimes}}, 1,\mathbin{\dot{\multimap}})$ on $\Meas$ in the above ``normal form''.
Then $\Gamma_x = (\overline{x} \mathbin{\dot{\otimes}} Y) \circ \inverse{\lambda_Y}$ and $\Gamma_y = (X \mathbin{\dot{\otimes}} \overline{y}) \circ  \inverse{\rho_X}$, where $\overline{x} \colon 1 \to X$ and $\overline{y} \colon 1 \to Y$ are the elements $\ast \mapsto x$ and $\ast \mapsto y$ respectively.
Here $\lambda_X \colon I \mathbin{\dot{\otimes}} X \cong X$ and $\rho_X \colon X \mathbin{\dot{\otimes}} I \cong X$ is the left and right unitors in the symmetric monoidal closed structure $(\mathbin{\dot{\otimes}}, 1,\mathbin{\dot{\multimap}})$ respectively.
Therefore, these graph functions $\Gamma_x$ and $\Gamma_y$ are measurable.
This implies that the identity function on $|X|\times|Y|$ forms a measurable function $X \otimes Y \to X \mathbin{\dot{\otimes}} Y$ because the $\sigma$-algebra $\Sigma_{X \mathbin{\dot{\otimes}} Y}$ is coarser than $\Sigma_{X \otimes Y}$.
Also the identity function on $|X|\times|Y|$ forms a measurable function $X \mathbin{\dot{\otimes}} Y \to X \times Y$ because $\pi_1 = \rho_X \circ (X \mathbin{\dot{\otimes}} !_Y)$ and $\pi_2 = \lambda_Y \circ (!_X \mathbin{\dot{\otimes}} Y)$.
\section{The Uniqueness of Tensorial Strength}\label{Uniqueness_strength}
We first show the uniqueness of tensorial strength for any monad on $\Meas$ with respect to the canonical symmetric monoidal closed structure $(\otimes,1,\multimap)$ on $\Meas$.
We slightly relax the uniqueness of tensorial strength in \cite{MOGGI199155} as below.

Let $\mathbb{C}$ be a category.
An object $I$ in $\mathbb{C}$ is called a \emph{separator} if for any pair of arrows $f,g \colon X \to Y$ in $\mathbb{C}$, the equality $f = g$ holds when $f \circ e = g \circ e$ for each  $e \colon I \to X$.
\begin{lemma}[{\cite[Proposition 3.4]{MOGGI199155}, Modified}]\label{lemma:uniqueness_strength}
Consider a symmetric monoidal category $(\mathbb{C}, \otimes, I)$ whose tensor unit $I$ is a separator of $\mathbb{C}$ such that for any morphism $\overline{z} \colon I \to X \otimes Y$, there are $\overline{x} \colon I \to X$ and $\overline{y} \colon I \to Y$ satisfying $\overline{z} = (\overline{x} \otimes \overline{y}) \circ \inverse{\lambda_I}$.

If $T$ is a strong monad with respect to $(\mathbb{C}, \otimes, I)$ then its tensorial strength $\mathrm{st}^T \colon ({-})\otimes T({=}) \Rightarrow T({-} \otimes{ =})$ is determined uniquely by
\[
\mathrm{st}^T_{X,Y} \circ (\overline{x} \otimes \overline{\xi}) \circ \inverse{\lambda_I} =  T((\overline{x} \otimes Y)\circ\inverse{\lambda_{Y}}) \circ \overline{\xi}
\] 
where $\overline{x} \colon I \to X$ and $\overline{\xi} \colon I \to TY$.
\end{lemma}
\begin{proof}
From the naturality of $\mathrm{st}^T$ and $\lambda$ and bifunctoriality of $\otimes$, 
we obtain, 
\begin{align*}
\mathrm{st}^T_{X,Y} \circ (\overline{x} \otimes \overline{\xi}) \circ \inverse{\lambda_I}
&= \mathrm{st}^T_{X,Y} \circ (\overline{x} \otimes TY) \circ (I \otimes \overline{\xi}) \circ \inverse{\lambda_I}\\
&= T(\overline{x} \otimes Y) \circ \mathrm{st}^T_{I,Y} \circ \inverse{\lambda_{TY}} \circ \overline{\xi} \circ \lambda_{I} \circ \inverse{\lambda_I}\\
&= T((\overline{x} \otimes Y)\circ \inverse{\lambda_Y}) \circ \overline{\xi}
\end{align*}
for any pair $\overline{x} \colon I \to X$ and $\overline{\xi} \colon I \to TY$.
Since any arrow $\overline{z} \colon I \to X \otimes TY$ is written as $\overline{z} = (\overline{x} \otimes \overline{\xi}) \circ \inverse{\lambda_I}$ for some $\overline{x} \colon I \to X$ and $\overline{\xi} \colon I \to TY$, the arrow $\mathrm{st}^T_{X,Y}$ is determined uniquely for each $X$ and $Y$.
\end{proof}
Any well-pointed cartesian monoidal category $(\mathbb{C}, \times, 1)$ satisfies the assumption of this lemma because the terminal object $1$ is a generator and $\lambda_X = \pi_2$ and $\inverse{\lambda_X} = \lrangle{!_X, \mathrm{id}}$.

The symmetric monoidal closed categories $(\Meas,\otimes,1,\multimap)$ and $(\mathbf{Top},\otimes,1,\multimap)$
discussed in Section \ref{SMCC_Meas} satisfy the assumption of the lemma because the terminal object $1$ is both tensor unit and generator, and each element $\overline{z} \colon 1 \to X \otimes Y$ ($\ast \mapsto (x,y)$) is obviously decomposed into a pair of elements $\overline{x} \colon 1 \to X$ ($\ast \mapsto x$) and $\overline{y} \colon 1 \to Y$ ($\ast \mapsto y$).
\section{The Giry Monad is not Strong}\label{Giry_not_Strong}
We show that the Giry monad $\mathcal{G}$ is not strong with respect to the canonical symmetric monoidal closed structure $(\otimes,1,\multimap)$ on $\Meas$.
In the following discussion, we consider the Giry monad $\mathcal{G}$, but we are able to prove that the subprobabilistic variant $\mathcal{G}_{\mathrm{sub}}$ is not strong in the same way.
\begin{theorem}
Giry monad $\mathcal{G}$ is not strong with respect to 
the canonical symmetric monoidal closed structure $(\otimes,1,\multimap)$ on $\Meas$.
\end{theorem}
Assume that the Giry monad $\mathcal{G}$ is strong with respect to the symmetric monoidal structure $(\Meas,\otimes,\multimap,1)$.
From \cite{Kock1972}, the strength $\mathrm{G}_{X,Y} \colon (X \multimap Y) \to (\mathcal{G}X \multimap \mathcal{G}Y)$ correspond bijectively to the tensorial strength $\mathrm{st}^{\mathcal{G}}_{X,Y} \colon X \otimes \mathcal{G}Y \to \mathcal{G}(X \otimes Y)$.
From the construction of $(\otimes,1,\multimap)$, we obtain $\mathrm{G}_{X,Y} = \lrceil{\mathcal{G}(\mathrm{ev}_{X,Y}) \circ \mathrm{st}^{\mathcal{G}}_{X \multimap Y,X}}$.
By Lemma \ref{lemma:uniqueness_strength}, the tensorial strength $\mathrm{st}^{\mathcal{G}}$ of $\mathcal{G}$ is determined uniquely by for any $x \in X$ and $\mu \in \mathcal{G}Y$, 
\[
\mathrm{st}^{\mathcal{G}}_{X,Y}(x,\mu)
= \mathcal{G}({(\overline{x} \otimes Y) \circ \inverse{\lambda_Y}}) \circ \overline{\mu}
= \mu(\inverse{((\overline{x} \otimes Y) \circ \inverse{\lambda_Y})}(-))
= \mu(\inverse{\Gamma_{x}}(-)).
\]
Hence, the following calculation shows that the strength $\mathrm{G}_{X,Y}$ is uniquely determined by the mapping that takes $f \colon X \to Y$, and returns $\mathcal{G}f \colon \mathcal{G}X \to \mathcal{G}Y$:
\[
\mathcal{G}(\mathrm{ev}_{X,Y}) \circ \mathrm{st}^{\mathcal{G}}_{X \multimap Y,X} (f,\mu)
= \mu(\inverse{\mathrm{ev}_{X,Y} \circ \Gamma_{f}}(-))
= \mu(\inverse{f}(-))
= \mathcal{G}(f)(\mu).
\]
However, as we show below, the component $\mathrm{G}_{X,Y}$ is not even a measurable function of type $(X \multimap Y) \to (\mathcal{G}X \multimap \mathcal{G}Y)$ for some $X$ and $Y$.
Hence, the Giry monad is not strong with respect to the canonical symmetric monoidal structure.
\subsection{Non-measurability of $\mathrm{G}_{X,Y}$}\label{Giry_not_Meas_functor}

We recall that $\Sigma_{\mathcal{G}X \multimap \mathcal{G}Y}$ is generated by
$\lrAngle{\mu,\inverse{\mathrm{ev}_U}(A)}$ for parameters $\mu \in \mathcal{G}X$, $U \in \Sigma_{Y}$, and $A \in \Sigma_{[0,1]}$.
We thus have, 
\[
\inverse{\mathrm{G}_{X,Y}}\lrAngle{\mu,\inverse{\mathrm{ev}_U}(A)}
= \Set{ f \in \Meas(X,Y)| \mu(\inverse{f}(U)) \in A}.
\]
Hence, the $\sigma$-algebra $\Set{\inverse{\mathrm{G}_{X,Y}}(K) | K \in \Sigma_{\mathcal{G}X \multimap \mathcal{G}Y}}$ of the inverse images of measurable subsets of $\mathcal{G}X \multimap \mathcal{G}Y$ along $\mathrm{G}_{X,Y}$ is at least finer than or equal to $\Sigma_{X \multimap Y}$, because for any $x \in X$ and $U \in \Sigma Y$, we obtain
\[
\lrAngle{ x, U } = \Set{ f \in \Meas(X,Y) |x\in \inverse{f}(U)} = \inverse{\mathrm{G}_{X,Y}} \lrAngle{ \delta_x, \inverse{\mathrm{ev}_U}(\{1\})}.
\]
\subsubsection{A $\sigma$-algebra $\Omega_{X,Y}$}
Consider measurable spaces $X$ and $Y$ whose underlying sets are infinite.
We define the family $\Omega_{X,Y}$ of all subsets of the form
\[
\lrangle{h, V} = \Set{f \in \Meas(X,Y) | \lrangle{ f(h(n)) }_{n \in \mathbb{N}} \in V} 
\]
where $h \colon \mathbb{N} \to X$ is an arbitrary injection and $V \subseteq {|Y|}^{\mathbb{N}}$ is an arbitrary subset.
\begin{lemma}
The collection $\Omega_{X,Y}$ forms a $\sigma$-algebra over $\Meas(X,Y)$ finer than $\Sigma_{X \multimap Y}$.
\end{lemma}
\begin{proof}
We have $\emptyset = \lrangle{ h, \emptyset } \in \Omega_{X,Y}$ where $h$ is an arbitrary injection.
For any $\lrangle{ h, V }\in \Omega_{X,Y}$, we have $\Meas(X,Y) \setminus \lrangle{ h, V }  =  \lrangle{ h,  {|Y|}^{\mathbb{N}} \setminus V } \in \Omega_{X,Y}$.
For any countable family $\{\lrangle{ h_m, V_m }\}_{m \in \mathbb{N}}$ with $\lrangle{ h_m, V_m } \in \Omega_{X,Y}$, we obtain $\bigcup_{m \in \mathbb{N}} \lrangle{ h_m, V_m } = \lrangle{ h, V }$ in the following steps:
\begin{enumerate}
\item The image $I = \Set{h_m(n)| m,n \in \mathbb{N}}$ is countably infinite, hence there is a bijection $k \colon \mathbb{N} \to I$.
Now we define $k_m = \inverse{k} \circ h_m$ for each $m \in \mathbb{N}$ and $h = \iota \circ k$ where $\iota \colon I \rightharpoonup X$ is the inclusion.
Since $(k \circ k_m)(n) = h_m(n)$ for all $m , n \in \mathbb{N}$, the injection $h$ and the family $\{k_m\}_{m\in\mathbb{N}}$ satisfy $h \circ k_m = h_m$ for each $m \in \mathbb{N}$.
\item We take the projection $\pi_l \colon {|Y|}^{\mathbb{N}} \to Y$ ($\lrangle{x_L}_{L\in\mathbb{N}} \mapsto x_l$) for each $l \in\mathbb{N}$ and the tuple $\lrangle{\pi_{k_m(n)}}_{n\in\mathbb{N}} \colon {|Y|}^{\mathbb{N}} \to {|Y|}^{\mathbb{N}}$ indexed by $\{k_m(n)\}_{n \in \mathbb{N}}$ for each $m \in \mathbb{N}$.
Then the inverse image $W_m = \inverse{\lrangle{\pi_{k_m(n)}}_{n\in\mathbb{N}}}(V_m)$ satisfies $\lrangle{h_m,V_m} = \lrangle{h, W_m}$ for each $m \in \mathbb{N}$.
\item We have $\bigcup_{m \in \mathbb{N}} \lrangle{ h_m, V_m } = \bigcup_{m \in \mathbb{N}} \lrangle{ h, W_m } =  \lrangle{ h, \bigcup_{m \in \mathbb{N}}\! W_m }$ (thus $V = \bigcup_{m \in \mathbb{N}}\! W_m $).
\end{enumerate}
Hence, $\Omega_{X,Y}$ is indeed a $\sigma$-algebra over $\Meas(X,Y)$.

For each $x \in X$ and $U \in \Sigma_Y$, we have $\lrAngle{ x, U } = \lrangle{ h, \inverse{\pi_0}(U) }$ where $h$ is an arbitrary injection such that $h(0) = x$.
From the minimality of  $\Sigma_{X \multimap Y}$, the $\sigma$-algebra $\Omega_{X,Y}$ is finer than $\Sigma_{X\multimap Y}$.
\end{proof}
The inclusion $\Sigma_{X \multimap Y} \subseteq \Omega_{X,Y}$ implies that 
for each measurable set $K \in \Sigma_{X \multimap Y}$, the membership $f \in K$ is determined by checking outputs $f(x_0), f(x_1), \ldots$ for some \emph{countable sequence} $x_0, x_1, \ldots$ of inputs.
\subsubsection{A Counterexample}
\begin{theorem}
Let $X = Y = (\mathbb{R},\mathcal{B}(\mathbb{R}))$, and let $\mu \in \mathcal{G}X$ be absolutely continuous with respect to the Borel measure $m$ (i.e. $m(A) = 0 \implies \mu(A) = 0$ for any $A \in \Sigma_{X}$).
We then obtain $\inverse{\mathrm{G}_{X,Y}} \lrAngle{ \mu, \inverse{\mathrm{ev}_{\{0\}}}(\{1\}) } \notin \Sigma_{X \multimap Y}$.
\end{theorem}
\begin{proof}
We write $K = \inverse{\mathrm{G}_{X,Y}} \lrAngle{ \mu, \inverse{\mathrm{ev}_{\{0\}}}(\{1\})}$.
We assume $K = \lrangle{ h, U} \in \Omega_{X,Y}$ holds for some $h$ and $U$.
We then have $U \neq \mathbb{R}^{\mathbb{N}}, \emptyset$ because $K$ is neither the whole space $\Meas(X,Y)$ nor the empty function space.
Hence, there is a pair of sequences $\lrangle{ s_n }_{n \in \mathbb{N}} \in \mathbb{R}^{\mathbb{N}}$ and $\lrangle{ t_n }_{n \in \mathbb{N}} \in \mathbb{R}^{\mathbb{N}}$ such that $\lrangle{ s_n }_{n \in \mathbb{N}} \in U$ and $\lrangle{ t_n }_{n \in \mathbb{N}} \notin U$.


We consider a measurable function  $f \in \Meas(X,Y)$.
We give the functions $f_1$ and $f_2$ by replacing the output of $f$ at each $h(n)$ to $s_n$ and $t_n$ ($n \in \mathbb{N}$) respectively, that is, 
\[
f_1 = f + \sum_{n \in \mathbb{N}}(s_n - f(h(n))) \cdot \chi_{\{h(n)\}}, \qquad
f_2 = f + \sum_{n \in \mathbb{N}}(t_n - f(h(n))) \cdot \chi_{\{h(n)\}}.
\]
Here, $\chi_{\{h(n)\}}$ is the indicator function of the closed (hence Borel) subset $\{h(n)\}$, and hence it is measurable for each $n \in \mathbb{N}$.
Since $\Meas(X,Y)$ is the set of Borel measurable functions on $\mathbb{R}$, and hence it is closed under scalar multiplication and countable addition, the functions $f_1$ and $f_2$ are measurable.
We obtain $f_1 \in K$ and $f_2 \notin K$ since $\lrangle{ s_n }_{n \in \mathbb{N}} \in U$ and $\lrangle{ t_n }_{n \in \mathbb{N}} \notin U$.
However, $f_1 \in K \iff f_2 \in K$ must hold, because $\mu(\Set{ x  |f(x) \neq f_1(x)}) = \mu(\Set{ x  |f(x) \neq f_2(x)}) = 0$ is obtained from the absolute continuity of $\mu$ with respect to the Borel measure $m$.
This is a contradiction.
Hence, there is no $h$ and $U$ such that $K = \lrangle{ h, U} \in \Omega_{X,Y}$.
From the construction of $\Omega_{X,Y}$, $K \notin \Omega_{X,Y}$.
Thus, $K \notin \Sigma_{X \multimap Y}$.

\end{proof}
\section{Concluding Remarks}
The proof that the Giry monad is strong with respect to the canonical symmetric monoidal closed structure $(\otimes,1,\multimap)$ in the preprint \cite{2016arXiv160102593S} has the following error:
The statement of \cite[Theorem 3.1]{2016arXiv160102593S} is just the naturality of $\mathrm{st}^{\mathcal{G}{\times}}_{X,Y} \circ \mathrm{id_{|X|\times|Y|}} \colon X \otimes \mathcal{G}Y \to  \mathcal{G}(X \times Y)$. Here we remark $\mathrm{st}^{\mathcal{G}{\times}}_{X,Y}$ is the tensorial strength for $\mathcal{G}$ with respect to the cartesian product on $\Meas$, and $\mathrm{id_{|X|\times|Y|}}$ obviously forms a symmetric monoidal natural transformation $X \otimes Y \to X \times Y$.
However, the above statement is mistaken for the existence of the natural transformation of the type $X\otimes \mathcal{G}Y \to \mathcal{G}(X \otimes Y)$ in the proof of existence of the tensorial strength for the Giry monad $\mathcal{G}$ with respect to the canonical symmetric monoidal closed structure $(\otimes,1,\multimap)$.

If there is a symmetric monoidal closed structure $(\dot{\otimes},1,\dot{\multimap})$ on $\Meas$ with respect to which makes the Giry monad strong, then there is a strong symmetric monoidal functor $U$ from $(\dot\otimes,1,\dot{\multimap})$ to the canonical symmetric monoidal closed structure $(\otimes,1,\multimap)$.
Moreover, by converting to the ``normal form'' discussed in the last two paragraphs of Section \ref{SMCC_Meas}, the $\sigma$-algebra $\Sigma_{X  \dot{\otimes} Y}$ of the space $X \dot{\otimes} Y$ satisfies $\Sigma_{X \times Y} \subsetneq \Sigma_{X  \dot{\otimes} Y} \subsetneq \Sigma_{X  \otimes Y}$.
We have not found yet an intermediate symmetric monoidal closed structure on $\Meas$ which is intermediate between the cartesian monoidal structure and the canonical symmetric monoidal closed structure.
\section*{Acknowledgement}
The author thanks Shin-ya Katsumata for constructive discussions and stimulating suggestions and Masahito Hasegawa for advice that contributed to improving the writing of this paper.
\bibliographystyle{plain}
\bibliography{reference}
\end{document}